\documentclass[12pt]{amsart}

 \usepackage{amsfonts,graphics,amsmath,amsthm,amsfonts,amscd,amssymb,amsmath,latexsym,multicol}
\usepackage{epsfig,url}
\usepackage{flafter}

\makeatletter

\def\jobis#1{FF\fi
  \def\predicate{#1}%
  \edef\predicate{\expandafter\strip@prefix\meaning\predicate}%
  \edef\job{\jobname}%
  \ifx\job\predicate
}

\makeatother

\if\jobis{proposal}%
\else
\fi

 \usepackage[matrix, arrow]{xy}

\DeclareMathOperator{\Supp}{Supp}

\DeclareMathOperator{\var}{var}

 \newcommand{\N}{\mathbb N}
 
 \newcommand{\Q}{\mathbb Q}

 \newcommand{\bir}{\dashrightarrow}

 \numberwithin{equation}{subsection}
 \numberwithin{footnote}{subsection}

 \newtheorem{cor}[subsection]{Corollary}

 \newtheorem{thm}[subsection]{Theorem}
 \newtheorem{conj}[subsection]{Conjecture}

{

}

 \newcommand{\ke}[1]{$\acute{\mbox{e}}$}
 \newcommand{\ku}[1]{$\acute{\mbox{u}$}}
 \newcommand{\kl}[1]{$\acute{\mbox{l}}$}
 \newcommand{\kh}[1]{$\acute{\mbox{h}}$}
 \newcommand{\kr}[1]{$\acute{\mbox{r}}$}
 \newcommand{\kx}[1]{$\acute{\mbox{x}}$}
 \newcommand{\ki}[1]{${\^\i}$}

\title{Iitaka conjecture $C_{n,m}$ in dimension six}
\author{Caucher Birkar}
\date{\today}
\begin{document}
\maketitle

\begin{abstract}
We prove that the Iitaka conjecture $C_{n,m}$  for algebraic fibre spaces holds up to dimension $6$, that is, when $n\le 6$.
\end{abstract}


\section{Introduction}

We work over an algebraically closed field $k$ of 
characteristic zero. Let $X$ be a normal variety. The canonical divisor $K_X$ is one of 
the most important objects associated with $X$ especially in birational geometry. 
If another normal variety $Z$ is in some way related to $X$, it is often crucial to find a relation 
between $K_X$ and $K_Z$. A classical example is when $Z$ is a 
smooth prime divisor on a smooth $X$ in which case we have $(K_X+Z)|_Z=K_Z$. 

An algebraic fibre space is a surjective morphism $f\colon X\to Z$ 
of normal projective varieties, with connected fibres. A central problem in birational geometry  
is the following conjecture which relates the 
Kodaira dimensions of $X$ and $Z$. In fact, it is an attempt to relate $K_X$ and $K_Z$.

\begin{conj}[Iitaka]\label{c-iitaka}
Let $f\colon X\to Z$ be an algebraic fibre space where $X$ and $Z$ are smooth projective varieties of dimension $n$ and $m$, 
respectively, and let $F$ be a general fibre of $f$. Then, 
$$
\kappa(X)\ge \kappa(F)+\kappa(Z)
$$
\end{conj}

This conjecture is usually denoted by $C_{n,m}$. A strengthend version was proposed by 
Viehweg (cf. [\ref{Viehweg}]) as follows which is denoted by $C_{n,m}^+$.

\begin{conj}[Iitaka-Viehweg]\label{c-iitaka-viehweg}
Under the assumptions of \ref{c-iitaka}, 
$$
\kappa(X)\ge \kappa(F)+\max\{\kappa(Z),\var(f)\}
$$
when $\kappa(Z)\ge 0$.
\end{conj}

 Kawamata [\ref{Kaw3}] showed that these conjectures  hold if the general fibre 
$F$ has a good minimal model, in particular, if the minimal model and the 
abundance conjectures hold in dimension $n-m$ for varieties of nonnegative Kodaira dimension. 
However, at the moment the minimal model conjecture for such varieties is known only up to dimension $5$ 
[\ref{B2}] and the abundance conjecture up to dimension $3$ [\ref{Miyaoka}][\ref{Kaw-abundance}] and some cases 
in higher dimensions which will be discussed below. Viehweg [\ref{Viehweg}] proved $C_{n,m}^+$ when $Z$ is of general type. When $Z$ is a curve $C_{n,m}$ was setteld by Kawamata [\ref{kaw2}]. Koll\'ar [\ref{kollar}] proved $C_{n,m}^+$ when $F$ is of general type. The latter also follows from Kawamata [\ref{Kaw3}] and the  existence of good minimal models for varieties of general type by Birkar-Cascini-Hacon-M$^c$Kernan [\ref{BCHM}].  We refer the reader to Mori [\ref{Mori}] for a detailed survey of the above conjectures and related problems. In this paper, we prove the following 

\begin{thm}\label{6-iitaka}
Iitaka conjecture $C_{n,m}$ holds when $n\le 6$.
\end{thm}

\begin{thm}\label{p-surface}
Iitaka conjecture $C_{n,m}$ holds when $m=2$ and  $\kappa(F)=0$.
\end{thm}

When $n\le 5$ or when $n=6$ and $m\neq 2$, $C_{n,m}$ follows immediately from theorems of Kawamata and deep results of the 
minimal model program. 

Iitaka conjecture is closely related to the following 

\begin{conj}[Ueno]\label{c-ueno}
Let $X$ be a smooth projective variety with $\kappa(X)=0$. Then, the Albanese map $\alpha\colon X\to A$ 
satisfies the following\\\\
$(1)$ $\kappa(F)=0$ for the general fibre $F$,\\
$(2)$ there is an etale cover $A'\to A$ such that $X\times_A A'$ is birational to $F\times A'$ over $A$.
\end{conj}

Ueno conjecture is often referred to as Conjecture K. Kawamata [\ref{kaw-albanese}] showed that $\alpha$ is an algebraic fibre space. See Mori [\ref{Mori}, \S 10] for a discussion of this conjecture. 

\begin{cor}\label{6-ueno}
Part $(1)$ of Ueno conjecture holds when $\dim X\le 6$.
\end{cor}
\begin{proof}
Immediate by Theorem \ref{6-iitaka}.
\end{proof}

Concerning part $(1)$ of Ueno conjecture, recently Chen and Hacon [\ref{CH}] showed that $\kappa(F) \le \dim A$.

\section*{Acknowledgements}

I would like to thank Burt Totaro for many helpful conversations and comments. I am grateful to Fr\'ed\'eric Campana for reminding me 
of a beautiful theorem of him and Thomas Peternell which considerably simplified my arguments.

\section{Preliminaries}

\emph{Nef divisors.} 
A Cartier divisor $L$ on a projective variety $X$ is called nef if $L\cdot C\ge 0$ 
for any curve $C\subseteq X$. If $L$ is a $\Q$-divisor, we say that it is nef if 
$lL$ is Cartier and nef for some $l\in \N$. We need a theorem about nef $\Q$-divisors 
due to Tsuji [\ref{Tsuji}] and Bauer et al. [\ref{nef}]. 

\begin{thm}
Let $L$ be a nef $\Q$-divisor on a normal projective variety $X$. Then, there is a dominant almost regular rational map $\pi\colon X\bir Z$ 
 with connected fibres to a normal projective variety, called the reduction map of $L$, such that\\\\ 
(1) if a fibre $F$ of $\pi$ is projective and $\dim F=\dim X-\dim Z$, then $L|_F\equiv 0$,\\
(2) if $C$ is a curve on $X$ passing  through a very general point $x\in X$ with $\dim \pi(C)>0$, then $L\cdot C>0$. 
\end{thm} 

Here by almost regular we mean that some of the fibres of $\pi$ are projective and away from the indeterminacy locus of $\pi$. 
Using the previous theorem, one can define the nef dimension $n(L)$ of the nef $\Q$-divisor $L$ to be $n(L):=\dim Z$. 
In particular, if $n(L)=0$, the theorem says that $L\equiv 0$.\\

\emph{Minimal models.} 
Let $X$ be a smooth projective variety. A projective variety $Y$ with terminal singularities is called a minimal model of $X$ 
if there is a birational map $\phi\colon X\bir Y$, such that $\phi^{-1}$ does not contract divisors, $K_Y$ is nef, and finally there is a common 
resolution of singularities $f\colon W\to X$ and $g\colon W\to Y$ such that $f^*K_X-g^*K_Y$ is effective and its support contains the birational transform of any prime divisor 
on $X$ which is exceptional over $Y$. If in addition $lK_Y$ is base point free 
for some $l\in \N$, we call $Y$ a good minimal model. 

The minimal model conjecture asserts 
that every smooth projective variety has a minimal model or a Mori fibre space, in particular, if the variety has 
nonnegative Kodaira dimension then it should have a minimal model. The abundance conjecture states that every minimal model is a good one.\\ 
 
\emph{Kodaira dimension}. Campana and Peternell [\ref{CP}] made the following interesting conjecture. 

\begin{conj}
Let $X$ be a smooth projective variety and suppose that $K_X\equiv A+M$ where $A$ and $M$ are effective and pseudo-effective 
$\Q$-divisors respectively. Then, $\kappa(X)\ge \kappa(A)$.
\end{conj} 

They proved the conjecture in case $M\equiv 0$ [\ref{CP}, Theorem 3.1]. This result is an important ingredient of the 
proofs below.\\

\section{Proofs}

\begin{proof}(of Theorem \ref{p-surface})
We are given that the base variety $Z$ has dimension $2$ and that $\kappa(F)=0$.
We may assume that $\kappa(Z)\ge 0$ otherwise the theorem is trivial. 
Let $p\in \N$ be the smallest number such that $f_*\mathcal{O}_X(pK_X)\neq 0$. By Fujino-Mori 
[\ref{FM}, Theorem 4.5], there is a diagram 
$$
\xymatrix{
X' \ar[d]^g\ar[r]^\sigma & X \ar[d]^f\\
Z' \ar[r]^\tau & Z
}
$$
in which $g$ is an algebraic fibre space of smooth projective varieties, $\sigma$ and $\tau$ are birational, and there are $\Q$-divisors $B$ 
and $L$ on $Z'$ and a $\Q$-divisor $R=R^+-R^-$ on $X'$ decomposed into its positive and 
negative parts satisfying the following:\\ 

(1) $B\ge 0$,\\

(2) $L$ is nef,\\

(3) $pK_{X'}=pg^*(K_{Z'}+B+L)+R$,\\

(4) $g_*\mathcal{O}_{X'}(iR^+)=\mathcal{O}_{Z'}$ for any $i\in\N$,\\

(5) $R^-$ is exceptional$/X$ and codimension 
of $g(\Supp R^-)$ in $Z'$ is $\ge 2$.

\vspace{0.5cm}

Thus for any sufficiently divisible $i\in\N$ we have\\
  
(6) $g_*\mathcal{O}_{X'}(ipK_{X'}+iR^-)=\mathcal{O}_{Z'}(ip(K_{Z'}+B+L))$\\

 If the nef dimension $n(L)=2$ or if $\kappa(Z)=\kappa(Z')=2$, then $ip(K_{Z'}+L)$ is big for some $i$ 
by Ambro  [\ref{Ambro}, Theorem 0.3]. So, $ip(K_{Z'}+B+L)$ is also big and by (6) 
and by the fact that $\sigma$ is birational and $R^-\ge 0$ is exceptional$/X$ we have 
$$
H^0(ipK_X)=H^0(ipK_{X'}+iR^-)=H^0(ip(K_{Z'}+B+L))
$$
for sufficiently divisible $i\in \N$. Therefore, in this case $\kappa(X)=2\ge \kappa(Z)$. 

 If  $n(L)=1$, then the nef reduction map $\pi\colon Z'\to C$ is regular where $C$ is a smooth projective curve, and there is a 
$\Q$-divisor $D'$ on $C$ such that $L\equiv \pi^*D'$ and $\deg D'>0$ by [\ref{nef}, Proposition 2.11]. On the 
other hand, if $n(L)=0$ then $L\equiv 0$. So, when $n(L)=1$ or $n(L)=0$, there is a $\Q$-divisor $D\ge 0$ such that 
$L\equiv D$.  Now letting $M:=\sigma_*g^*(D-L)$, for sufficiently divisible $i\in \N$, we have 
$$
H^0(ip(K_{X}+M))=H^0(ip(K_{X'}+g^*D-g^*L)+iR^-)
$$
$$
\hspace{1.6cm} =H^0(ip(K_{Z'}+B+D))
$$ 
and by Campana-Peternell [\ref{CP}, Theorem 3.1] 
$$
\kappa(X)\ge \kappa(K_{X}+M)=\kappa(K_{Z'}+B+D)\ge \kappa(Z)
$$
\end{proof}

\begin{proof}(of Theorem \ref{6-iitaka})
We assume that $\kappa(Z)\ge 0$ and $\kappa(F)\ge 0$ otherwise the theorem is trivial. 

If $m=1$, then the theorem follows from  Kawamata [\ref{kaw2}]. On the other hand, if $n-m\le 3$, then 
the theorem follows from Kawamata [\ref{Kaw3}] and the existence of good minimal models in dimension $\le 3$. So, 
from now on we assume that $n=6$ and $m=2$ hence $\dim F=4$.  By the flip theorem of Shokurov [\ref{shokurov}] and the termination theorem of Kawamata-Matsuda-Matsuki [\ref{KMM}, 5-1-15] $F$ has a minimal model (see also [\ref{B2}]). If $\kappa(F)> 0$, by Kawamata [\ref{Kaw}, Theorem 7.3] such a minimal model is good, so we can apply [\ref{Kaw3}] again. Another possible argument would be to apply Koll\'ar [\ref{kollar}] when $F$ is of general type and to use 
the relative Iitaka fibration otherwise. 

Now assume that $\kappa(F)=0$. In this case, though we know that $F$ has a minimal model, abundance is not yet known. Instead, we use Theorem \ref{p-surface}.
\end{proof}

\vspace{2cm}

\flushleft{DPMMS}, Centre for Mathematical Sciences,\\
Cambridge University,\\
Wilberforce Road,\\
Cambridge, CB3 0WB,\\
UK\\
email: c.birkar@dpmms.cam.ac.uk

\end{document}